\documentclass[final,1p,times]{elsarticle}

\usepackage{amsfonts,mathrsfs,ulem,dsfont,amsmath,amsthm,latexsym,amscd,bbm,amssymb,graphicx,paralist,epsfig,graphics,exscale}

\usepackage[colorlinks=true]{hyperref}
\usepackage{color,soul}

\newtheorem*{thmA}{Theorem A}
\newtheorem*{thmB}{Proposition B}
\newtheorem{thm}{Theorem}
\newtheorem{lem}{Lemma}
\newtheorem*{thmGW}{Theorem}

\theoremstyle{definition}
\newtheorem{defn}[thm]{Definition}

\newtheorem{oq}{Question}
\numberwithin{equation}{section}
\def\C{\boldsymbol{C}}
\def\X{\boldsymbol{X}}
\def\x{\boldsymbol{x}}
\def\U{\boldsymbol{U}}

\journal{JDE (Received May 31 2013; revised 23 November 2014)}
\begin{document}

\begin{frontmatter}

\title{Chaotic dynamics of continuous-time topological semi-flows on Polish spaces}

 \author{Xiongping Dai}
 \address{Department of Mathematics, Nanjing University, Nanjing 210093, People's Republic of China}
 \ead{xpdai@nju.edu.cn}
\begin{abstract}
Differently from Lyapunov exponents, Li-Yorke, Devaney and others that appeared in the literature, we introduce the concept, \textit{chaos}, for a continuous semi-flow $f\colon\mathbb{R}_+\times X\rightarrow X$ on a Polish space $X$ with a metric $d$, which is useful in the theory of ODE and is invariant under topological equivalence of semi-flows. Our definition is weaker than Devaney's one since here $f$ may have neither fixed nor periodic elements; but it implies \textit{repeatedly observable sensitive dependence on initial data}: there is an $\epsilon>0$ such that for any $x\in X$, there corresponds a \textit{dense $G_\delta$-set} $\mathcal{S}_\epsilon^u(x)$ in $X$ satisfying
\begin{equation*}
\limsup_{t\to+\infty}d(f^t(x),f^t(y))\ge\epsilon\quad \forall y\in\mathcal{S}_\epsilon^u(x).
\end{equation*}
This sensitivity is obviously stronger than Guckenheimer's one that requires only $d(f^t(x),f^t(y))\ge\epsilon$ for some moment $t>0$ and some $y$ arbitrarily close to $x$.
\end{abstract}
\end{frontmatter}















\section{Introduction}\label{sec1}

Lyapunov stability and chaotic behavior of motions $f(t,x)$ of a dynamical system on a metric space $(X,d)$ with continuous-time $0\le t<\infty$ have been the most fascinating aspects in nonlinear science.

The Lyapunov stability of a motion $f(t,x)$ is defined from A.A.~Markov as follows: given any $\epsilon>0$ there exists a number $\delta>0$ so that for any $y$ within a distance $\delta$ of $x$, $d(f(t,x),f(t,y))<\epsilon$ for all $t>0$; see \cite[Definition~V.8.03]{NS}. Although there have been various definitions of chaos for discrete-time dynamical systems since Li and Yorke 1975, yet as far as we know there is no rigorous mathematical definition of chaos for continuous-time integrated flow or semi-flow of a differential system on Riemannian manifolds like $\mathbb{R}^n$. In fact, there are essential differences between the chaotic dynamics of discrete-time and continuous-time dynamical systems. For example, if $X$ has no isolated points and if a \textit{topologically transitive} continuous map $T$ of $X$ has a \textit{fixed} or \textit{periodic} points, then $T$ is Li-Yorke chaotic; that is, one can find an \textit{uncountable scrambled set} $S\subset X$ for $T$ such that for all $x,y\in S$ with $x\not=y$,
\begin{equation*}
\liminf_{n\to+\infty}d(T^n(x),T^n(y))=0\quad\textrm{and}\quad
\limsup_{n\to+\infty}d(T^n(x),T^n(y))>0
\end{equation*}
by the work of W.~Huang and X.~Ye 2002~\cite{HY}. However, this is not the case for continuous-time flows as shown by a counterexample as follows: if $f\colon\mathbb{R}\times\mathds{T}^1\rightarrow\mathds{T}^1$ is the continuous-time rotation of the unit circle $\mathds{T}^1$, then it is topologically transitive and the underlying space $\mathds{T}^1$ exactly consists of a periodic orbit, but does not have any Yi-Yorke's scrambled pair $(x,y)$.

In this paper, we will introduce a kind of chaos which may characterize the opposite side of stability---instability: if the motion $f(t,x)$ is chaotic in a region $\varSigma$, then there exists an $\epsilon>0$ such that for any $\delta>0$ one can find some $y\in\varSigma$ within a distance $\delta$ of $x$ such that $\limsup_{t\to+\infty}d(f(t,x),f(t,y))\ge\epsilon$.
That is to say, $f(t,x)$ has the \textit{sensitive dependence on initial condition}. Importantly, this sensitivity is repeatedly observable; since we can choose a Borel subset $\mathscr{O}$ of $[0,\infty)$ with Lebesque measure $\infty$ such that $d(f(t,x),f(t,y))\ge\epsilon$ for all $t\in\mathscr{O}$.

More precisely, let $f\colon\mathbb{R}\times\mathbb{R}^n\rightarrow\mathbb{R}^n$ be the integrated flow of a complete $\mathrm{C}^1$-vector field $\X$ on the $n$-dimensional euclidean space $\mathbb{R}^n$; i.e.,
\begin{equation*}
\frac{d}{dt}f(t,x)=\X(f(t,x)),\quad \forall t\in\mathbb{R}\textrm{ and }x\in\mathbb{R}^n.
\end{equation*}
To study the unpredictability of the dynamical behavior of a given Lagrangian stable motion $f(t,\x)$ as $t\to+\infty$, the classical way in statistical mechanics is to see if the Lyapunov exponent, defined as
\begin{equation*}
\chi(f,\x)=\limsup_{t\to+\infty}\frac{1}{t}\log\bigg{\|}\frac{\partial f(t,\x)}{\partial x}\bigg{\|},
\end{equation*}
is positive or not. Differently from this analytic method, we present a topological formulation of the unpredictability of the motion $f(t,\x)$.

Let $\varSigma_{\x}$ be the closure of the forward orbit $f([0,\infty),\x)$ of the motion $f(t,\x)$, for any $\x\in\mathbb{R}^n$. We say $f(t,\x)$ is \textit{chaotic} in $\varSigma_{\x}$ if the following two topological conditions are satisfied:
\begin{itemize}
\item (Minimal-sets density) the (Birkhoff) recurrent motions are dense in $\varSigma_{\x}$;
\item (Non-minimality) $\varSigma_{\x}$ is not minimal.
\end{itemize}
Then we can obtain the following sensitivity theorem:
\begin{thmA}
If the Lagrangian stable motion $f(t,\x)$ is chaotic in $\varSigma_{\x}$, then it is observably sensitive in the sense that there exists an $\epsilon>0$ having the property: for any $\delta>0$, one can find some $y\in\varSigma_{\x}$ with $\|\x-y\|<\delta$ and a Borel subset $\mathscr{O}=\mathscr{O}(y)$ of $[0,\infty)$ with Lebesque measure $\infty$, such that $\|f(t,\x)-f(t,y)\|>\epsilon$ for every $t\in\mathscr{O}$.
\end{thmA}

We note here that since the $\mathrm{C}^1$-vector field $\X$ is bounded restricted to $\varSigma_{\x}$, to prove Theorem~A it is sufficient to prove the following weak result:
\begin{thmB}
If the Lagrangian stable motion $f(t,\x)$ is chaotic in $\varSigma_{\x}$, then there exists an $\epsilon>0$ such that for any $\delta>0$, one can find some $y\in\varSigma_{\x}$ with $\|\x-y\|<\delta$ so that $\limsup_{t\to+\infty}\|f(t,\x)-f(t,y)\|>\epsilon$.
\end{thmB}

Consequently if the motion $f(t,\x)$ is chaotic, then it is not Lyapunov asymptotically stable. In fact, we will prove this proposition in a more general framework; see Theorem~\ref{thm4} in Section~\ref{sec2}.

As shown by the motion of the irrational rotation on the $2$-dimensional torus $\mathds{T}^2$, the weak non-minimality condition is not ignorable for our chaos. In addition, the minimal-sets density condition is not ignorable too.

\subsection{Topological semi-flows}\label{sec1.1}
We now formalize our study framework in the sequel of this paper. Let $X$ be a Polish space (i.e., a complete separable metric space) with a metric $d$, and let $\mathbb{R}_+=[0,\infty)$ be the nonnegative time axis of the dynamical systems we consider here. From now on we let
\begin{equation*}
f\colon \mathbb{R}_+\times X\rightarrow X;\quad (t,x)\mapsto f(t,x)=f^t(x)
\end{equation*}
be a continuous semi-flow on $X$; that is, $f(t,x)$ satisfies the following three conditions.
\begin{enumerate}
\item[(I)] The initial condition: $f(0,x)=x$ for all $x\in X$.
\item[(II)] The condition of continuity with respect to the variables $t$ and $x$: if there be given two convergent sequences $t_n\to t_0$ in $\mathbb{R}_+$ and $x_n\to x_0$ in $X$, then $f(t_n,x_n)\to f(t_0,x_0)$ as $n\to\infty$.
\item[(III)] The semigroup condition: $f(t_2,f(t_1,x))=f(t_1+t_2,x)$ for any $x\in X$ and any times $t_1,t_2\in\mathbb{R}_+$.
\end{enumerate}
Then the pair $(X,f)$ is named as a \textit{continuous-time topological dynamical system}. This framework is applicable for the theory of ordinary differential equations.

From condition (II) there is obtained as a corollary the following property:
\begin{enumerate}
\item[(II)$^\prime$] For any $q\in X$, any number $T>0$ (arbitrarily large) and any $\varepsilon>0$ (arbitrarily small), there can be found a number $\delta=\delta(q,T,\varepsilon)>0$ such that if $d(q,x)<\delta$ and $0<t\le T$, then $d(f(t,q),f(t,x))<\varepsilon$.
\end{enumerate}

\begin{proof}
If this property were false, then there could be found a sequence of points $q_n\to q$ in $X$ and a corresponding sequence of numbers $t_n\to t_0$ in $[0,T]$, such that
$d(f(t_n,q),f(t_n,q_n))\ge\alpha$ for some constants $\alpha>0$. Then by the triangle inequality,
$$d(f(t_n,q),f(t_n,q_n))\le d(f(t_n,q),f(t_0,q))+d(f(t_0,q),f(t_n,q_n)).$$
Hence by condition (II), it follows that $d(f(t_0,q),f(t_0,q))\ge\alpha$, a contradiction. This proves the assertion.
\end{proof}

As usual, a point $p\in X$ is called a \textit{periodic point of period $\pi$} of $f$ if $f^\pi(p)=p$ and $f^t(p)\not=p$ for all $0<t<\pi$, for some number $\pi>0$; by $\textrm{Per}(f)$ we mean the set of all periodic points of $f$. If $f^t(p)=p$ for all $t>0$ then $p$ is called a \textit{fixed point} of $f$ and we write $\textrm{Fix}(f)$ for the set of all fixed points of $f$.

We recall that $f$ is \textit{topologically transitive} if there is an (strictly positive) orbit of $f$
\begin{equation*}
\mathcal{O}_{\varepsilon+}(f,x):=\left\{f^t(x)\,|\,\varepsilon\le t<\infty\right\},\quad \textrm{for some }\varepsilon>0,
\end{equation*}
that is dense in $X$; i.e., $\overline{\mathcal{O}_{\varepsilon+}(f,x)}=X$. According to a well-known theorem of Birkhoff, there follows that
\begin{itemize}
\item $(X,f)$ is topologically transitive if and only if for every pair of nonempty open subsets $U$ and $V$ of $X$, there follows $V\cap f^t(U)\not=\varnothing$ for some $t\le-1$. See, e.g., \cite[Theorem~5.9]{Wal} for the discrete-time $\mathbb{Z}$-action case.
\end{itemize}

\begin{proof}
Let $(X,f)$ be topologically transitive with $\overline{O_{\varepsilon+}(f,x)}=X$ and $U,V$ two non-void open subsets of $X$. Then one can find $0<t_1<t_1+2<t_2$ such that $f^{t_1}(x)\in V$ and $f^{t_2}(x)\in U$, and thus $f^{t_1-t_2}(U)\cap V\not=\varnothing$.

Conversely, assume that whenever $U,V$ are nonempty open sets then there exists $t\le-1$ with $V\cap f^t(U)\not=\varnothing$. Let $U_1,U_2,\dotsc$ be a countable base for $X$. Then $\left\{x\in X\,|\,\overline{\mathcal{O}_{1+}(f,x)}=X\right\}=\bigcap_{n=1}^\infty\bigcup_{-\infty<t\le-1}f^t(U_n)$ and $\bigcup_{-\infty<t\le-1}f^t(U_n)$ is clearly dense by the hypothesis. Hence the result follows.
\end{proof}

We now introduce a kind of weaker transitivity.

\begin{defn}
The topological semi-flow $f\colon \mathbb{R}_+\times X\rightarrow X$ is said to be \textit{topologically quasi-transitive} if there is an (positive) orbit of $f$,
$\mathcal{O}_{+}(f,x):=\left\{f^t(x)\,|\,0\le t<\infty\right\}$, which is dense in $X$, i.e., $\overline{\mathcal{O}_{+}(f,x)}=X$.
\end{defn}

It should be noted that since here every samples $f(t,\cdot)\colon X\rightarrow X$ are not necessarily surjective, the above two transitivity properties are different. For example, the closure of a positive orbit itself is topologically quasi-transitive, but not topologically transitive except it is Poissonwise recurrent. The essential difference is that if $f$ is topologically transitive there is a residual set of transitive points; but this is not the case for topological quasi-transitivity.

As usual, a point $p\in\textrm{Fix}(f)\cup\textrm{Per}(f)$ is called a \textit{critical element} of $f$ and its corresponding closed orbit $\mathcal{O}_+(f,p)$ a \textit{critical orbit} of $f$.

Recall from~\cite[V.7.05]{NS} that a motion $f(t,x)$ is called \textit{recurrent} in the sense of Birkhoff if for any $\varepsilon>0$ there exists a $T=T(\varepsilon)>0$ such that for any two moments $\tau,\upsilon\in\mathbb{R}_+$ there can be found a number $t$ with $\tau<t<\tau+T$ such that $d(f^\upsilon(x),f^t(x))<\varepsilon$. Clearly the fixed and periodic motions $f(t,p)$ both are Birkhoff recurrent. We note here that a Birkhoff recurrent point is sometimes called an almost periodic point or a uniformly recurrent point for the discrete-time dynamical systems in most of the literature; see, e.g., \cite{GH} and \cite{Wal}.

Since the underlying space $X$ of our topological dynamical system $(X,f)$ is complete, the connection between recurrent motions and minimal sets is established by the following theorem of Birkhoff:
\begin{itemize}
\item A motion $f(t,x)$ is Birkhoff recurrent if and only if $\overline{\mathcal{O}_+(f,x)}$ is a compact minimal invariant set of $f$; see, e.g., \cite[Theorems~V.7.06 and V.7.07]{NS}.
\end{itemize}

It is easy to show that a Birkhoff recurrent motion is Poisson stable and Lagrangian stable.\footnote{A motion $f(t,x)$ is called Poisson stable if $x$ itself is an $\omega$-limit point of $f(t,x)$, i.e. for some sequence $t_n\uparrow\infty$ we have $f(t_n,x)\to x$; and it is said to be Lagrangian stable if $\overline{\mathcal{O}_+(f,x)}$ is compact.} In addition, we say two motions $f(t,x)$ and $f(t,y)$ are \textit{far away} each other if $\overline{\mathcal{O}_+(f,x)}\bigcap\overline{\mathcal{O}_+(f,y)}=\varnothing$. So any two distinct critical orbits are far away each other.

\subsection{Definition of chaos of topological semi-flow}\label{sec1.2}
To describe the unpredictability and disorder of the dynamical behavior of the continuous-time dynamical system $(X,f)$, differently from Li-Yorke~\cite{LY} and Devaney~\cite{Dev} but similar to Glasner and Weiss 1993~\cite{GW}, we now introduce a type of chaos stated as follows:

\begin{defn}\label{def2}
The continuous-time topological dynamical system $(X,f)$ is called \textit{chaotic}, provided that the following three topological conditions are satisfied:
\begin{enumerate}
\item[(1)] (Topological quasi-transitivity) $f$ is topologically quasi-transitive on $X$.
\item[(2)] (Minimal-sets density) the Birkhoff recurrent points of $f$ are dense in $X$.
\item[(3)] (Non-minimality) there are at least two motions of $f$ far away each other; or equivalently, $(X,f)$ is not minimal itself.
\end{enumerate}
\end{defn}

This definition is mainly motivated by Robert L.~Devaney's chaos~\cite{Dev} that requires the following additional condition:
\begin{enumerate}
\item[(4)] $\overline{\textrm{Fix}(f)\cup\textrm{Per}(f)}=X$.
\end{enumerate}
We should note here that condition (1) together with (2) implies that
\begin{enumerate}
\item[$(1)^\prime$] $f$ is topologically transitive on $X$.
\end{enumerate}
However our definition of chaos is essentially weaker than Devaney's, because here our topological semi-flow $f$ possibly has neither any fixed points nor periodic points in $X$.
If there holds the unnecessarily restrictive condition (4) as in Devaney's chaos, then our condition (2) is satisfied trivially.

We note here that if $X$ has no isolated points, conditions $(1)^\prime$ and (4) imply condition (3) for any discrete-time dynamical system. However this is not the case for our continuous-time context. For example, if $f$ is the continuous-time rotation of the unit circle $\mathds{T}^1$ as mentioned before; conditions $(1)^\prime$, (2) and (4) all then hold except condition (3).

From the above Definition~\ref{def2}, it is easy to see that our chaos is preserved under topological equivalence of semi-flows. That is to say, if $g\colon\mathbb{R}_+\times Y\rightarrow Y$ is another topological dynamical system and $h\colon X\rightarrow Y$ is a homeomorphism such that $h$ transports orbits of $f$ onto orbits of $g$, i.e., $h(\mathcal{O}_+(f,x))=\mathcal{O}_+(g,h(x))$ for every $x\in X$, then $(Y,g)$ is chaotic if and only if so is $(X,f)$.

\subsection{Observable sensitivity}\label{sec1.3}
For a continuous transformation $T\colon X\rightarrow X$, the topological transitivity together with the density of periodic points of $T$ implies the interesting property:
\begin{itemize}
\item \textit{Guckenheimer sensitivity}~\cite{Gu}: there exists a positive $\epsilon$ such that for all $x\in X$ and all neighborhoods $U$ of $x$, there is some $y\in U$ and for some integer $n>0$, $d(T^n(x),T^n(y))\ge\epsilon$. That is to say, $T$ is not Lyapunov stable at every point $x\in X$. See, e.g., \cite{BB, Sil, GW} for different proofs of it and \cite{KM} for an extension to $C$-semigroup actions.
\end{itemize}
However, for a mechanical system $(X,f)$, one often expects that an important phenomenon can be observed repeatedly; in other words, it is expected that to the continuous transformation $T\colon X\rightarrow X$, for ``sufficiently many'' initial values $y$ there are ``sufficiently many'' times $n_k$ such that $d(T^{n_k}(x),T^{n_k}(y))\ge\epsilon$ for all $k$.

For this, we now introduce a type of observable sensitivity, which is the most interesting part of our chaos described by Definition~\ref{def2}.

\begin{defn}\label{def3}
The topological dynamical system $(X,f)$ is said to have \textit{sensitive dependence on initial data}, provided that one can find a \textit{sensitive constant} $\epsilon>0$ such that for any $x\in X$, there exists a dense $G_\delta$-set $\mathcal{S}_\epsilon^u(x)$ in $X$ so that
\begin{equation*}
\limsup_{t\to+\infty}d(f^t(x),f^t(y))\ge\epsilon
\end{equation*}
for each $y\in\mathcal{S}_\epsilon^u(x)$.
\end{defn}

Here $\mathcal{S}_\epsilon^u(x)$ is a large set from the topological point of view. So according to Definition~\ref{def3}, for any $\varDelta>0$ and any $y\in\mathcal{S}_\epsilon^u(x)$ we can find moments $t\ge\varDelta$ such that $d(f^t(x),f^t(y))>\epsilon$. This shows that our sensitivity is observable repeatedly. This is just the opposite side of the Lyapunov asymptotical stability.

Although our chaos definition looks weaker than Devaney's chaos, it still implies the central idea in chaos---sensitive dependence on initial data in the sense of Definition~\ref{def3}. This will be proved in Section~\ref{sec2}; see Theorem~\ref{thm4} below.

Here we further ask the following question: If $(X,f)$ is chaotic, does the set of $\epsilon$-observable times,
$$\mathscr{O}(y)=\left\{t>0\,|\,d(f^t(x),f^t(y))\ge\epsilon\right\}$$
for any $\epsilon,x$ and $y$ as in Definition~\ref{def3}, have the Lebesque measure $\infty$?

As we pointed out before, if $f(t,x)$ is the integrated flow of a bounded $\mathrm{C}^1$-vector field, this holds trivially. Generally, if $X$ is compact, then $\mathscr{O}(y)$ has the Lebesque measure $\infty$.

Similarly to Li-Yorke scrambled set, by $\mathcal{S}^u$ we mean the \textit{$u$-scrambled} Borel subset of $(X,f)$ such that
\begin{equation*}
\limsup_{t\to+\infty}d(f^t(x),f^t(y))>0\quad \forall x,y\in\mathcal{S}^u\textrm{ with }x\not=y.
\end{equation*}
We will prove that $\mathcal{S}^u$ is uncountable and dense in $X$ if $(X,f)$ is chaotic in the sense of Definition~\ref{def2}; see Theorem~\ref{thm5} in Section~\ref{sec3}.

\subsection{Completely non-chaotic system}\label{sec1.4}
Let $\Lambda$ be a subset of $X$; then $(\Lambda,f)$ is called a subsystem of $(X,f)$ if $\Lambda$ is an $f$-invariant closed subset of $X$. The system $(X,f)$ is said to be \textit{completely non-chaotic} if it does not have any subsystems that are chaotic in their right in the sense of Definition~\ref{def2}. In Section~\ref{sec4}, we will show that if $(X,f)$ is not completely non-chaotic, then there is at least one maximal chaotic subsystem; see Theorem~\ref{thm7} below.

\subsection{Li-Yorke's chaos}\label{sec1.5}
From the interesting work of Huang and Ye 2002~\cite{HY}, it follows that for any continuous transformation $T$ from $X$ into itself, Devaney's chaos implies Li-Yorke's chaos; that is, one can find an \textit{uncountable scrambled} set $S\subseteq X$ in the sense that
\begin{equation*}
\liminf_{n\to+\infty}d(T^n(x),T^n(y))=0\quad \textrm{and}\quad\limsup_{n\to+\infty}d(T^n(x),T^n(y))>0
\end{equation*}
for any $x,y\in S$ with $x\not=y$.

We now end this introductory section with the following

\begin{oq}
If the continuous-time topological dynamical system $(X,f)$ is chaotic in the sense of Definition~\ref{def2}, does it appear the Li-Yorke chaotic phenomenon in $X$?
\end{oq}

In Huang and Ye's proof, the periodic point and topological transitivity both play important roles. However, in our context, although $f$ is topologically transitive by $(1)^\prime$, here $f$ does not need to have any periodic points or fixed points. So there needs a new idea for this question.

\section{Sensitive dependence on initial data}\label{sec2}
Let $f\colon\mathbb{R}_+\times X\rightarrow X$ be the continuous-time topological semi-flow on the Polish space $(X,d)$ as before. Recall from Definition~\ref{def3} that $(X,f)$ is said to have sensitive dependence on initial data if there is a constant $\epsilon>0$ such that to every $x\in X$ there is a dense $G_\delta$-set $\mathcal{S}_\epsilon^u(x)$ of $X$ such that
$\limsup_{t\to+\infty}d(f^t(x),f^t(y))\ge\epsilon$ for all $y\in\mathcal{S}_\epsilon^u(x)$.
Here by a $G_\delta$-set in $X$ we mean a set which can be expressed as the intersection of countable many open sets of $X$.

Using statistical property of a recurrent motion, importantly there follows the sensitive dependence on initial data from our definition of chaos as follows.

\begin{thm}\label{thm4}
Let $(X,f)$ be chaotic in the sense of Definition~\ref{def2}. Then there follows the sensitive dependence on initial data in the sense of Definition~\ref{def3}.
\end{thm}

\begin{proof}
Since $X$ contains at least two motions of $f$ far away each other from condition (3) of Definition~\ref{def2}, one can find a number $\delta_0>0$ such that for all $\hat{x}\in X$ there exists a corresponding motion, say $f(t,q_{\hat{x}})$, not necessarily recurrent but dependent of $\hat{x}$, such that
\begin{equation*}
d\left(\hat{x},\overline{\mathcal{O}_+(f,q_{\hat{x}})}\right)\ge\delta_0,
\end{equation*}
where $d(\hat{x},A)=\inf_{a\in A}d(\hat{x},a)$ for any subset $A$ of $X$. We will show that $f$ has the sensitive dependence on initial data with sensitivity constant $\epsilon=\delta_0/12$ in the sense of Definition~\ref{def3}.

Write simply $\delta=\delta_0/4$ and define for any $x\in X$ the Borel set
\begin{equation*}
\C_\epsilon^u(x)=\left\{y\in X\,\big{|}\,\limsup_{t\to+\infty}d(f^t(x),f^t(y))>\epsilon\right\}.
\end{equation*}
Next we will prove that $\C_\epsilon^u(x)$ is dense in $X$ for each $x\in X$.

For this, we let $x,\hat{x}$ be two arbitrary points in $X$ and let $U$ be an arbitrary neighborhood of $\hat{x}$ in $X$.

Since the Birkhoff recurrent motions of $(X,f)$ are dense in $X$ from condition (2) of Definition~\ref{def2}, there exists a Birkhoff recurrent point $p\in U\cap B_{\delta/2}(\hat{x})$, where $B_r(\hat{x})$ is the open ball of radius $r$ centered at $\hat{x}$ in $X$. As we noted above, there must exist another point $q=q_{\hat{x}}\in X$ whose orbit $\mathcal{O}_+(f,q)$ is of distance at least $4\delta$ from the given point $\hat{x}$.

Let $\eta>0$ be such that $\eta<\delta/2$. Then from the Birkhoff recurrence of the motion $f(t,p)$, it follows that one can find a constant $T=T(\eta,p)>0$ such that for any $\gamma\ge0$, there is some moment $t_\gamma\in[\gamma,\gamma+T)$ verifying that
\begin{equation*}
d(p,f^{t_\gamma}(p))<\eta.
\end{equation*}
We simply write
\begin{equation*}
V=\bigcap_{t\in[0,2T)}f^{-t}(B_\delta(f^t(q))),\quad \textrm{where }f^{-t}(\cdot)={f(t,\cdot)}^{-1}\colon X\rightarrow X.
\end{equation*}
Clearly from condition (II)$^\prime$ of the topological semi-flow $f$, it follows that $V$ is a neighborhood of $q$ in $X$ but not necessarily open, and it is nonempty since $q\in V$.

Since $f$ is topologically quasi-transitive on $X$ by condition (1) of Definition~\ref{def2}, there exists at least one point $z\in U\cap B_\delta(\hat{x})$ such that $f^N(z)\in V$ for some sufficiently large number $N\gg T$. Let
$$
N=jT-r,\quad \textrm{where }0\le r<T\textrm{ and }j\in\mathbb{N},
$$
and choose by $\gamma=jT$
$$
t_{jT}\in[jT,(j+1)T)\quad \textrm{such that }d(p,f^{t_{jT}}(p))<\eta.
$$
Then $0\le t_{jT}-N<2T$.

By the above construction, one has
$$
f^{t_{jT}}(z)=f^{t_{jT}-N}(f^N(z))\in f^{t_{jT}-N}(V)\subseteq B_\delta(f^{t_{jT}-N}(q)).
$$
From the triangle inequality of metric and $d(p,f^{t_{jT}}(p))<\eta$, it follows that
\begin{equation*}\begin{split}
d(f^{t_{jT}}(p),f^{t_{jT}}(z))&\ge d(p,f^{t_{jT}}(z))-d(p,f^{t_{jT}}(p))\\
&\ge d(\hat{x},f^{t_{jT}}(z))-d(p,\hat{x})-\eta\\
&\ge d\left(\hat{x},f^{t_{jT}-N}(q)\right)-d\left(f^{t_{jT}-N}(q),f^{t_{jT}}(z)\right)-d(p,\hat{x})-\eta.
\end{split}\end{equation*}
Consequently, since $\eta\le\delta/2$, $p\in B_{\delta/2}(\hat{x})$ and $f^{t_{jT}}(z)\in B_\delta(f^{t_{jT}-N}(q))$, it holds that
$$
d\left(f^{t_{jT}}(p),f^{t_{jT}}(z)\right)\ge2\delta.
$$
Therefore from the triangle inequality of metric again, one can obtain either
$$
d(f^{t_{jT}}(\hat{x}),f^{t_{jT}}(z))\ge\delta
$$
or
$$
d(f^{t_{jT}}(\hat{x}),f^{t_{jT}}(p))\ge\delta.
$$
Repeating this argument for another likewise $N$ bigger than $(j+2)T$, one can find a sequence $t_n=j_nT\uparrow+\infty$ as $n\to+\infty$ such that
either $d(f^{t_n}(\hat{x}),f^{t_n}(z))\ge\delta$ or $d(f^{t_n}(\hat{x}),f^{t_n}(p))\ge\delta$,
for all $n\ge1$. Thus in either case, we have found a point $\hat{y}\in U$ such that
$$
\limsup_{t\to+\infty}d(f^t(\hat{x}),f^t(\hat{y}))\ge\delta=3\epsilon.
$$
Using the triangle inequality once more, we see either
\begin{equation*}
\limsup_{t\to+\infty}d(f^t(x),f^t(\hat{y}))>\epsilon
\end{equation*}
or
\begin{equation*}
\limsup_{t\to+\infty}d(f^t(x),f^t(\hat{x}))>\epsilon.
\end{equation*}
Since $\hat{x}, U$ both are arbitrary and $\hat{y}\in U$, $\C_\epsilon^u(x)$ is dense in $X$.

Finally for any integer $n\ge1$ and any $x\in X$, let
$$
W_\epsilon^u(x,n)=\left\{y\in X\,|\,d(f^t(x),f^t(y))>\epsilon\textrm{ for some }t\in[n,\infty)\right\}.
$$
Since $f(t,x)$ is continuous from condition (II) in Section~\ref{sec1.1}, $W_\epsilon^u(x,n)$ is open in $X$.
Now for any $x\in X$ we let
\begin{equation*}
\mathcal{S}_\epsilon^u(x)=\bigcap_{k=1}^\infty\left(\bigcup_{n=k}^\infty W_\epsilon^u(x,n)\right).
\end{equation*}
Then $\mathcal{S}_\epsilon^u(x)$ is a $G_\delta$-set in $X$. In addition, it is easy to see that ${\C}_\epsilon^u(x)\subseteq\mathcal{S}_\epsilon^u(x)$. Therefore, $\mathcal{S}_\epsilon^u(x)$ is a dense $G_\delta$-set in $X$.

This thus completes the proof of Theorem~\ref{thm4}.
\end{proof}

The above proof of Theorem~\ref{thm4} has been motivated by the surprising work of Banks et al.~\cite{BB} for Devaney's chaos in the discrete-time dynamical system case. However, without the density of periodic points and the topological transitivity of $f$, we here need some essential improvements of the proof of \cite{BB}.

If $(X,f)$ is topologically transitive, \textit{not sensitive on initial conditions in the sense of Guckenheimer}, and $X$ has no isolated point, then for every $\varepsilon>0$ there exist a transitive point $x_0\in X$ and a neighborhood $\U$ of $x_0$ such that for all $y\in \U$, we have $d(f^t(x_0),f^t(y))\le\varepsilon$ for every $t>0$ \cite[Lemma~1.1]{GW}. This \textit{equi-continuity} of $(f^t)_{t>0}$ at $x_0$ is the key point for Glasner and Weiss in \cite{GW} to prove the Guckenheimer sensitivity of Devaney's chaos using ergodic approach. However, in our situation, there is an obstruction for us to employ Glasner and Weiss' ergodic approach:
When $(X,f)$ is not sensitive in the sense of our Definition~\ref{def3}, we cannot obtain the equi-continuity of $(f^t)_{t>0}$ at any transitive points, because for every $y\in \U$ we can only get $d(f^t(x_0),f^t(y))\le\varepsilon$ for $t\ge t(y,\varepsilon)$. Here $t(y,\varepsilon)$ is not necessarily uniform with respect to $y\in\U$.

In \cite{GW}, Glasner and Weiss introduced a very general measure-theoretic condition:
\begin{enumerate}
\item[(5)] There exists an $f$-invariant probability measure on $X$, which is positive on every non-empty open set.
\end{enumerate}
Then following the same argument as in \cite[Theorem~1.3]{GW}, from the syndeticity of return times we can easily obtain the continuous-time version of Glasner and Weiss' theorem:

\begin{thmGW}[Glasner and Weiss]
Let $(X,f)$ be a continuous-time topological semi-flow. If there hold the conditions $(1)^\prime, (5)$ and $(3)$, then $(X,f)$ has the sensitivity on initial conditions in the sense of Guckenheimer; that is, there exists a constant $\epsilon>0$ such that for all $x\in X$ and all neighborhoods $U$ of $x$, there is some $y\in U$ and for some $t>0$, $d(f^t(x),f^t(y))\ge\epsilon$.
\end{thmGW}

Recall that a point $x\in X$ is called \textit{regular} for the semi-flow $f$ if it is a generic point of some $f$-invariant Borel probability measure $\mu$ with $\mu(U)>0$ for every open neighborhood $U$ of $x$ in $X$. Clearly, if condition (5) holds, then the regular points are dense in $X$. However since $X$ is not necessarily compact, the converse is not necessarily true.
We now ask naturally the following at the end of this section.

\begin{oq}
If the continuous-time semi-flow $(X,f)$ satisfies conditions (1) and (3) such that
the regular points are dense in $X$,
does it have the sensitive dependence on initial data in the sense of Definition~\ref{def3} or Guckenheimer?
\end{oq}
\section{On the $u$-scrambled set}\label{sec3}
Throughout this section, let $(X,f)$ be a chaotic continuous-time topological semi-flow in the sense of Definition~\ref{def2}, where the state space $X$ is a Polish space as in Section~\ref{sec1.1}. Clearly, $X$ has no isolated points from Definition~\ref{def2}.

For a relation $R\subset X\times X$ and $x\in X$, write $R(x)=\{y\in X\,|\,(x,y)\in R\}$. Then we need a lemma.

\begin{lem}[{Huang and Ye~\cite[Lemma~3.1]{HY}}]\label{lem1}
If $R$ is a symmetric relation with the property that there is a dense $G_\delta$ subset $A$ of $X$ such that for each $x\in A$, $R(x)$ contains a dense $G_\delta$ subset of $X$, then there is a dense uncountable subset $B$ of $X$ such that $(x,y)\in R$ whenever $x,y\in B$ with $x\not=y$.
\end{lem}

As a simple consequence of the statements of Theorem~\ref{thm4} and Lemma~\ref{lem1} above, we can easily obtain the following.

\begin{thm}\label{thm5}
For the chaotic dynamical system $(X,f)$, its $u$-scrambled set $\mathcal{S}^u$ is uncountable and dense in $X$.
\end{thm}

\begin{proof}
Let the symmetric relation $R$ be defined by $(x,y)\in R$ if and only if $\limsup_{t\to+\infty}d(f^t(x),f^t(y))>0$. Then from Theorem~\ref{thm4} proved in Section~\ref{sec2}, it follows that for any $x\in X$, $R(x)$ contains a dense $G_\delta$ subset of $X$ since $\mathcal{S}_\epsilon^u(x)\subseteq R(x)$ for some $\epsilon>0$. So, the statement comes immediately from Lemma~\ref{lem1}.

This completes the proof of Theorem~\ref{thm5}.
\end{proof}
\section{Maximal chaotic subsystems}\label{sec4}
In applications, it often appears that the continuous-time topological dynamical system $(X,f)$ itself is not chaotic in the sense of Devaney and even of Definition~\ref{def2}, but actually there are chaotic subsystems such as Axiom A flows. For that, we now introduce the following concept.

\begin{defn}\label{def6}
Let $\Lambda$ be an $f$-invariant closed subset of $X$. If $(\Lambda,f)$ itself is chaotic in the sense of Definition~\ref{def2} and moreover there is no other likewise chaotic subset that properly contains $\Lambda$, then $\Lambda$ is called a \textit{maximal chaotic subset}/{\it subsystem} of $(X,f)$.
\end{defn}

This section will be devoted to proving the following existence theorem of maximal chaotic subsystem.

\begin{thm}\label{thm7}
If $(X,f)$ is a continuous-time topological semi-flow which is not completely non-chaotic, then there always exists a maximal chaotic subsystem of $(X,f)$.
\end{thm}

\begin{proof}
Let $\mathscr{D}$ be the family of all chaotic subsets of $(X,f)$. It is nonempty, since $(X,f)$ is not completely non-chaotic. We now equip $\mathscr{D}$ with a partial order as follows: $A\preceq B$ if and only if $A\subseteq B$, for any $A,B\in\mathscr{D}$. Let $\mathscr{C}$ be any given totally ordered chain of $(\mathscr{D},\preceq)$ and set
\begin{equation*}
F=\overline{{\bigcup}_{C\in\mathscr{C}}C}.
\end{equation*}
It is easy to see that $F$ as a subspace of $X$ is a complete separable metric space and that $F$ is $f$-invariant. We then claim that $F\in\mathscr{D}$.

In fact, we need only check that the semi-flow $f$ is topologically transitive restricted to $F$. For that, let $U$ and $V$ be two non-empty open subsets of $F$. Then one can find two open sets $U^\prime$ and $V^\prime$ of $X$ such that
$$
U=U^\prime\cap F\quad \textrm{and}\quad V=V^\prime\cap F.
$$
So there are at least two elements $C_1, C_2\in\mathscr{C}$ such that
$$
U^\prime\cap C_1\not=\varnothing\quad\textrm{and}\quad V^\prime\cap C_2\not=\varnothing.
$$
Since $\mathscr{C}$ is totally ordered under $\preceq$, we have either $C_1\subseteq C_2$ or $C_2\subseteq C_1$. Without loss of generality, we assume $C_1\subseteq C_2$. Then from $(1)^\prime$, it follows that
$$
U^{\prime\prime}:=U^\prime\cap C_2\not=\varnothing\quad\textrm{and}\quad V^{\prime\prime}:=V^\prime\cap C_2\not=\varnothing.
$$
As $(C_2,f)$ itself is chaotic from the definition of $\mathscr{D}$, it follows that
$$
V^{\prime\prime}\cap f^t(U^{\prime\prime})\not=\varnothing\quad \textrm{for some }t\ge1.
$$
Therefore $V\cap f^t(U)\not=\varnothing$ and so $f$ is topologically transitive restricted to $F$.

This thus completes the proof of Theorem~\ref{thm7}.
\end{proof}

\subsection*{\textbf{Acknowledgments}}%
This work was partly supported by National Natural Science Foundation of China grants $\#$11431012, 11271183 and PAPD of Jiangsu Higher Education Institutions.


\begin{thebibliography}{11}
\bibitem{BB}
   \newblock {J.~Banks, J.~Brooks, G.~Cairns, G.~Davis and P.~Stacey},
   \newblock {On Devaney's definition of chaos},
   \newblock {Amer. Math. Monthly 99 (1992), 332--334}.

\bibitem{Dev}
   \newblock {R.L.~Devaney},
   \newblock {An Introduction to Chaotic Dynamical Systems},
   \newblock {Addison-Wesley, Reading, MA, 1989}.

\bibitem{GW}
   \newblock {E.~Glasner and B.~Weiss},
   \newblock {Sensitive dependence on initial conditions},
   \newblock {Nonlinearity 6 (1993), 1067--1075}.

\bibitem{GH}
   \newblock {W.H.~Gottschalk and G.A.~Hedlund},
   \newblock {Topological Dynamics},
   \newblock {Amer. Math. Soc. Coll. Publ., Vol. 36, Amer. Math. Soc., Providence, RI, 1955}.

\bibitem{Gu}
   \newblock {J.~Guckenheimer},
   \newblock {Sensitive dependence on initial conditions for one-dimensional maps},
   \newblock {Comm. Math. Phys. 70 (1979), 133--160}.

\bibitem{HY}
   \newblock {W.~Huang and X.-D.~Ye},
   \newblock {Devaney's chaos or $2$-scattering implies Li-Yorke's chaos},
   \newblock {Topology $\&$ Appl. 117 (2002), 259--272}.

\bibitem{KM}
   \newblock {E.~Kontorovich and M.~Megrelishvili},
   \newblock {A note on sensitivity of semigroup actions},
   \newblock {Semigroup Forum 76 (2008), 133--141}.

\bibitem{LY}
   \newblock {T.~Li and J.~Yorke},
   \newblock {Period three implies chaos},
   \newblock {Amer. Math. Monthly 82 (1975), 985--992}.

\bibitem{NS}
  \newblock {V.V.~Nemytskii and V.V.~Stepanov},
  \newblock {Qualitative Theory of Differential Equations},
  \newblock {Princeton University Press, Princeton, NJ, 1960}.

\bibitem{Sil}
  \newblock {S.~Silverman},
  \newblock {On maps with dense orbits and the definition of chaos},
  \newblock {Rocky Mt. J. Math. 22 (1992), 353--375}.

\bibitem{Wal}
  \newblock {P.~Walters},
  \newblock {An Introduction to Ergodic Theory},
  \newblock {GTM vol.~79, Springer-Verlag, New York Heidelberg Berlin, 1982}.
\end{thebibliography}
\end{document}